\documentclass[12pt,leqno]{article}
\usepackage{latexsym}
 \usepackage[pagewise]{lineno}\linenumbers

\usepackage{appendix}
 
 \usepackage{esvect} 
 \usepackage{latexsym}
\usepackage{amsfonts}       
\usepackage{color}
\usepackage{amsmath}
\newcounter{eqnsave}

\def\C{{{{\rm {\mbox{\small l}}} \kern -.50em {\rm C}}}}
\def\I{{{{\rm l} \kern -.10em {\rm I}}}}
\def\R{{{{\rm l} \kern -.15em {\rm R}}}}
\def\N{{{{\rm l} \kern -.15em {\rm N}}}}
\def\E{{{{\rm l} \kern -.15em {\rm E}}}}

\newcommand{\bea}{\begin{eqnarray}}
\newcommand{\beas}{\begin{eqnarray*}}

\newcommand{\ba}{\begin{array}}
\newcommand{\ea}{\end{array}}

\newtheorem{theorem}{Theorem}[section]
\newtheorem{lemma}[theorem]{Lemma}

\newtheorem{remark}[theorem]{Remark}
\newtheorem{definition}{Definition}[section]
\newcommand{\eea}{\end{eqnarray}}                %
\newcommand{\bas}{\begin{eqnarray*}}            %
\newcommand{\eas}{\end{eqnarray*}}              %
\begin{document}
\title{On a comparison method for a parabolic-elliptic system of  chemotaxis  with density-suppressed motility and  logistic growth
}
 \author{ J.Ignacio Tello }
  \date{}
  
\maketitle
\begin{abstract}
We consider a parabolic-elliptic  system of partial differential equations  with  chemotaxis and logistic growth given by the system \newline 
$$ \left\{
\begin{array}{l} u_t -\Delta (u \gamma(v)= \mu u(1-u),   
\\ - \Delta v +v=u,  
\end{array} \right.
$$
 under Neumann boundary conditions and appropriate initial  data in a bounded and regular domain $\Omega$ of $\R^N$ (for $N \geq 1)$, where  $\gamma \in C^3([0, \infty))$ and satisfies
  \newline 
  
$\gamma (s) > 0$,    $\gamma^{\prime}(s) \leq 0$, $\gamma^{\prime \prime} (s) \geq 0$, $\gamma^{\prime \prime \prime}(s) \leq 0$  for any $s \geq 0$
$$-2 \gamma^{\prime}(s) + \gamma^{\prime \prime}(s)s \leq \mu_0< \mu$$
$$\frac{[\gamma^{\prime}(s)]^2}{\gamma(s)} \leq c, \quad  \mbox{ for any } s \in [0, \infty).
$$
We obtain the global existence and uniqueness of bounded in time solutions and the following asymptotic behavior \newline 
$$\|u- 1\|_{L^{\infty}(\Omega)} +\|v- 1\|_{L^{\infty}(\Omega)} \rightarrow 0, \quad \mbox{ when } t \rightarrow +\infty.$$
\end{abstract}



 \footnotetext[2]{ Departamento de Matem\'{a}ticas Fundamentales, Facultad de Ciencias, 
Universidad Nacional de Educaci\'on a Distancia, 28040 Madrid. Spain}

\section{Introduction}\label{s1}

Chemotaxis has been studied from a mathematical point of view in the last decades, especially significant are the  pionnering works of Keller and Segel   \cite{KS} and \cite{KS2} 
 modeling the phenomenon.
 The model presented in  \cite{KS} and \cite{KS2}
is a fully parabolic system of two equations involving   a chemical stimuli and a biological especies orientating its movement  in response to the  mentioned    stimuli.  

 There exists an  extensive mathematical literature studying chemotaxis  systems of Partial Differential Equations, see for instance the surveys   
   Horstmann \cite{horstmann}, \cite{horstmann2011},  Bellomo et al \cite{bellomo1}, Hillen and Painter \cite{hp}    and references therein
for more details. 
One of the main topics studied in the literature is under which assumptions the solution  blows up   or  it remains  bounded for any $t<{\infty}$. Such a dichotomy is also present in 
 chemotaxis systems with logistic growth terms. For instance,  for bounded and regular domains in $\R^N$,   the solution to the parabolic elliptic system 
$$\left\{ \begin{array}{l} u_t -\Delta u=- div(u \chi \nabla v)+  \mu u(1-u),  
\\ - \Delta v +v=u, 
\end{array} \right.
$$  under homogeneous Newmann boundary conditions,
globally exists and it is bounded in time when $$\chi <\frac{N-2}{N} \mu, \quad \mbox{ if }  N >2 \quad \mbox{ and } \mu>0, \ \mbox{  if } N=2$$ 
see Tello and Winkler \cite{tw1}. In Khan and Stevens \cite{khan},   the global existence of solutions is obtained for the limit case  $\chi =\frac{N-2}{N} \mu$.
 The result is also valid if the domain is    $\R^N$ and  $\chi= \mu,$  see 
 Salako and Shen
\cite{salako}.   Galakhov, Tello and Salieva   \cite{gst} study the system 
$$\left\{ \begin{array}{l} u_t -\Delta u=- div(u^m \chi \nabla v) + \mu u(1-u^{\kappa}),  
\\ - \Delta v +v=u^{\theta} ,  
\end{array} \right.
$$
in a bounded and regular domain with homogeneous Neumann boundary conditions. In \cite{gst}, the authors obtain that, under any of the  assumptions 
\begin{itemize}
\item[] $\theta>m+ \gamma-1,$
\item[]  $\theta =m+ \gamma -1 $  and $\mu > \frac{N\theta -2}{2(m-1)+N \theta}\chi,$
\end{itemize}
the solution exists globally in time for regular initial data. Moreover, if  
$$\theta \geq m+ \gamma -1 \mbox{   and } \quad  \mu> 2\chi ,$$ 
 and  the initial data satisfy  
 $ 0< \underline{u}_0 \leq 
 u_0\leq \overline{u}_0 < \infty$,  for some positive constants $\overline{u}_0$ and $\underline{u}_0$ 
we have the following asymptotic behavior 
$$\|u-1\|_{L^{\infty}(\Omega)}+ \|w-1\|_{L^{\infty}(\Omega)}\rightarrow 0 \quad \mbox{ as } t \rightarrow \infty.$$

Recently, in Liu et al \cite{liu},  the authors propose a fully parabolic system of two parabolic equations 
to model the pattern formation of e-coli bacteria (see also Fu et al \cite{fu}). In \cite{liu}, $u$ denotes the e-coli density and $v$ the molecule acyl-homoserine lactone concentration (AHL)  which is excreted  by the e-coly cells. The system proposed in \cite{liu}
reads as follows  
$$ \left\{ \begin{array}{l}  \displaystyle 
u_t - \Delta( \tilde{\gamma}(v) u) = \mu u (1-  u / \rho_s),  
\\ [2mm]  \displaystyle 
v_t - D_h\Delta v+ \alpha v  = \beta  u 
 \end{array} \right.
$$
 where $\tilde{\gamma}$ is given by a bounded  function 
 $$\tilde{\gamma} (v)= \frac{D_{\rho}+D_{\rho,0}  \frac{v^m}{K_h^m} }{1+\frac{v^m}{K_h^m}}
$$ for $m=20$ 
and  positive constants   $D_{\rho}$,   $D_{\rho,0}$,   $K_h$, 
$\mu $, $\rho_s$,
$D_h$,  
$\alpha$ and $\beta$. 
\newline 
Notice that the parabolic equation satisfied by $u$ in  \cite{liu}   can be written as  a particular case of  the classical Keller-Segel system  
$$ u_t - div( \tilde{\gamma}(v)\nabla u) = div(u \tilde{\gamma}^{\prime}(v)\nabla v) + \mu u (1- u / \rho_s),   $$
where $ \tilde{\gamma}$ represents the diffusion  coefficient and $\tilde{\gamma}^{\prime}(v)$ is the chemoatractant coefficient which are 
 clearly linked. 

In \cite{liu}, the diffusion coefficient of $v$ is taken $D_h\sim 400 \mu m^2 s^{-1}$ and $D_{\rho,0} / D_{\rho} <<1$. Considering such a range of values for $D_h$ and $D_{\rho,0} / D_{\rho}$ and the range of   data for $v$,  a natural simplification of the fully parabolic  system proposed in \cite{liu}  is to take $$\tilde{\gamma}(v)= \frac{D_{\rho}  }{1+\frac{v^m}{K_h^m}}$$ for $v$  satisfying the elliptic equation 
$$ -D_h \Delta v+ \alpha  v  =  \beta u .$$
Such simplification 
 transforms  the   fully parabolic
 problem into  a parabolic-elliptic system. Now we introduce the   rescaled variables and parameters 
$$ \tilde{u}= \frac{u}{\rho_{s}}, \quad \tilde{v}= \alpha v, \quad \tilde{x}= \frac{\alpha^{\frac{1}{2}}}{D_h^{\frac{1}{2}}}x, \quad \tilde{ \beta}= \rho_s \frac{\beta}{\alpha}   $$
 and the system becomes 
$$ \left\{ \begin{array}{l}  \displaystyle 
u_t - \Delta( \gamma(\tilde{v}) \tilde{u}) = \mu \tilde{u} (1- \tilde{u}),  
\\ [2mm] 
-\Delta \tilde{v}+  \tilde{v}  = \tilde{\beta } \tilde{u} . 
 \end{array} \right.
$$
For simplicity we drop the tilde,  assume that $\tilde{\beta}=1$ and complete the system with Neumann boundary conditions and appropriate initial data  in a bounded and regular  domain  $\Omega$  
  \begin{equation} \label{1.1} \left\{ \begin{array}{l}  \displaystyle 
u_t - \Delta( \gamma(v) u) = \mu u (1- u),  
\\ [2mm] 
-\Delta v+  v  = \tilde{\beta } u ,  
\\[2mm] 
u(0,x)=u_0(x).
\end{array} \right.
\end{equation} 
The fully parabolic system  for $\mu=0$ have been considered from a mathematical point of view in Tao and Winkler \cite{taowinkler2017}. In \cite{taowinkler2017}, the function 
 $\gamma$ belongs to  $ C^3([0, \infty))$ and satisfies   
$$k_0 \leq \gamma (s) \leq k_1, \quad | \gamma^{\prime }| \leq k_2, \quad \mbox{ for any } s\in [0, \infty)   $$
for some positive constants $k_i$ $(i=0 \cdots 2)$. The authors prove that the solution is uniformly bounded  when $\Omega$ is a two dimensional bounded and regular domain. 
If the dimension is bigger than $2$,   there exists a global-in-time weak solution in the appropriate Sobolev space. Moreover, for a range of parameters, the solution becomes classical provided $t>t_0$ for some $t_0<\infty.$
After \cite{taowinkler2017}, Jin, Kim and Wang \cite{jkw}  proved  that the fully parabolic system with logistic term possesses global   classical solutions   with a uniform-in-time bounds when
the limit 
$$\lim_{v\rightarrow \infty} \frac{\gamma^{\prime} (v)}{\gamma(v)}$$  
exists and $$\lim_{v\rightarrow \infty} {\gamma(v)}=0$$
in a two dimensional bounded domain.  
Moreover, if $\mu$ is large enough the unique positive constant steady state $u=v=1$,   is asymptotically stable. 

 In \cite{ahn}, the existence of global in time bounded solutions is proved for  the parabolic-elliptic system  (for $\mu=0$)  when 
 $\gamma(v)= v^{-\kappa}$ for any  $\kappa>0$ if $\Omega$ is a one or two-dimensional bounded domain  and for   $\kappa <\frac{2}{N-2}$ in a N-dimensional bounded domain for $N>2$.  
  The linear stability is also presented in these  cases.

  The system (\ref{1.1}) has been already studied  in 
 Fujie and Jiang  \cite{fj2020},  the authors consider the parabolic-elliptic system with logistic growth in a 2-dimensional bounded  domain $\Omega$. In that case 
 the solution is uniformly bounded provided 
 $$ \frac{|\gamma^{\prime}(s) |^2}{\gamma(s) }
 \leq k_0 <\infty \quad \mbox{ for any $s\geq 0$. } $$
 Moreover, the global existence is also obtained in  a  2-dimensional bounded domain  when $\mu=0$ and  $\gamma$ satisfies 
 $$\gamma^{\prime} \leq 0, \quad \lim_{s\rightarrow + \infty} s^k\gamma(s)=+ \infty$$  
 or $\gamma (s)= e^{-s}$ and the initial mass is small enough. In the last case, the asymptotic stability of the solution converging to the average of the initial data is also given.
  A complementary case,  for large initial mass and the same function $\gamma$, i.e.  $\gamma(s)= e^{-s}$ in the  unit ball,  produces blow up at $t= +\infty$.   
   
 In Fujie and Jiang \cite{fj2020B} the parabolic-elliptic system is also  considered for $\mu=0$. The authors obtain  the existence of a unique global classical solution which is uniformly-in-time bounded under the assumption
 $$\lim_{s \rightarrow \infty} e^{\alpha s} \gamma(s) = + \infty, \quad \mbox{ for any $\alpha>0$.} $$
 If the previous assumption is only satisfied for large $\alpha$, the solution also exists in time provided the initial data is small enough. The proofs are given in  bounded domains in arbitrary dimension $N$.
 
 The parabolic elliptic case is also studied in 
 Jiang \cite{jiang} 
  for  $\mu=0$ and   $\gamma$ satisfying 
 $$ \frac{N+2}{4} |\gamma^{\prime}(s)|^2 \leq \gamma (s) \gamma^{\prime \prime} (s), \qquad \mbox{ for $s \geq 0$.} $$
 In \cite{jiang},  the author proved that  the unique solution is global in time and  converges to the average of the initial data in $L^{\infty}(\Omega)$. 
 Moreover, if $\gamma= v^{-k}$ for $k$
  satisfying 
 $$\begin{array}{lcl} 
 k \in (0,1), & \mbox{ if} &  N =4,5 \\
 [2mm] 
 k \in (0, \frac{4}{N-2})& \mbox{ if} &  N  \geq 6
 \end{array} 
 $$
 the same result is also obtained. 
  
  In this article we study the solutions of equation (\ref{1.1})  in the following sense.
  \begin{definition}  \label{defu}  We say that
  $(u, v)  $ is a weak solution to (\ref{1.1}) if $$ u  \in L^2(0,T:H^2(\Omega)) \cap H^1(0,T:L^2(\Omega)) \cap  C(0,T:L^2(\Omega) ),$$
   $$v \in  C(0,T:H^1(\Omega)) $$ and for any $\phi \in 
  L^2(0,T:H^2(\Omega)) \cap H^1(0,T:L^2(\Omega)) \cap  C(0,T:L^2(\Omega) ) $ such that
  $$\frac{\partial \phi}{\partial \overrightarrow{n}}  =0 , \quad \mbox{ in }  x\in \partial \Omega $$ 
  we have that 
  $$\begin{array}{lcl} \displaystyle 
  -\int_0^T\int_{\Omega} u \phi_t dx dt  -  
  \int_0^T \int_{\Omega} u \gamma(v) \Delta \phi dxdt & = & \displaystyle \mu \int_0^T \int_{\Omega} u(1-u) \phi dxdt
 \\ [2mm]  && \displaystyle +\int_{\Omega}[ u_0 \phi_0-u(T) \phi(T)]dx
  \end{array}
  $$   
  and 
  $$ \int_{\Omega} \nabla v \nabla \phi dx+ \int_{\Omega}  v   \phi dx=   \int_{\Omega} u  \phi dx. 
  $$   
  \end{definition} 
The problem is studied under the assumptions 
\begin{equation} 
\label{H1} \gamma \in C^3( [0, \infty)), \quad  \gamma (s) \geq 0, \end{equation} 
\begin{equation} 
\label{H2}     \gamma^{\prime}(s) \leq 0, \qquad \gamma^{\prime \prime} (s) \geq 0, \qquad \gamma^{\prime \prime \prime}(s) \leq 0 , \quad \mbox{ for any $s \geq 0$}, 
\end{equation} 
\begin{equation} \label{H3}  -2 \gamma^{\prime}(s) + \gamma^{\prime \prime}(s)s \leq \mu_0< \mu.
\end{equation} 
\begin{equation} \label{H4} 
\frac{ |\gamma^{\prime}(s) |^2}{\gamma(s)} \leq c_{\gamma} <\infty , \quad \mbox{ for any $s\in [0, \infty),$} 
\end{equation} 
where the initial datum $u_0$ satisfies 
\begin{equation} 
\label{H5} 
u_0 \in C^{2, \alpha} (\overline{\Omega}) , \quad \frac{\partial u_0}{\partial\overrightarrow{n}}=0 \mbox{ in } \partial \Omega .
\end{equation} 
There exists positive constants $\overline{u}_0$, $\underline{u}_0$   such that 
\begin{equation} 
\label{H6} 
0 <\overline{u}_0 \leq u_0 \leq \overline{u}_0 <\infty.  
\end{equation} 
Notice that assumptions (\ref{H1})-(\ref{H4}) are satisfied for instance by 
$$\gamma_1(s):= e^{-\alpha_1 s}, \ \mbox{ for any } \alpha_1 >0, \quad \gamma_2(s):= [\epsilon_2+s]^{-\alpha_2} \ \mbox{ for any } \epsilon_2, \ \alpha_2 >0$$
and $\mu$ large enough. 
 
 \
 
The main result of the article is enclosed in the following theorem.
\begin{theorem} \label{t1.1} 
Under assumption (\ref{H1})-(\ref{H6}) there exists a unique solution $(u,v)$ in the sense of Definition \ref{defu}  in the time interval  $(0, \infty)$ 
satisfying 
$$\lim_{t \rightarrow \infty} \|u - 1\|_{L^{\infty}(\Omega)} +    \|v - 1\|_{L^{\infty}(\Omega)} =0.$$
\end{theorem} 

The article is organized as follows. In section \ref{s2} we present  the local in time existence of solutions. The proof  is obtained by  standard arguments using a priori estimates in the appropriate Sobolev spaces.
In Section \ref{s3} an auxiliary system of ODE is introduced, the solutions of such system are  used in Section \ref{s4} as sub and super solutions of the equation. Finally in Section \ref{s5} the proof of the asymptotic behavior of  the solution is presented as a consequence of the asymptotic behavior of the sub and super solutions. 

\

The comparison method used in the article is known as rectangle method,   similar  
comparison methods     have been already applied to  reaction-diffusion systems by different authors,  
see for instance Pao  \cite{pao2},  Conway and Smoller \cite{conway},   Fife and Tang \cite{fife} and also  Negreanu and Tello \cite{nt2}. 
  Parabolic-elliptic chemotaxis problems    are among the systems where the rectangle method have been  successfully applied, in this case,  the 
 sub- and super-solutions are defined as the solutions of   a coupled nonlinear 
ODE$^{\prime}$s systems, see for instance \cite{ft} and  \cite{nt}.  The method have been also applied to parabolic-elliptic chemotaxis systems of two species, see for instance 
Tello and Winkler     \cite{tw2} and  Stinner, Tello and Winkler \cite{stw1}.

Throughout the article we consider the constant $c_{\Omega}$ defined in the  following definition. 
\begin{definition} \label{dc} Let $\Omega$ be a regular bounded domain of  $\R^{N}$,  we define $c_{\Omega}$ and $c_p $ (for $p>N/2$) 
 as follows
$$  c_{\Omega}:= \sup_{f \in L^{\infty}(\Omega)} 
\left\{\frac{ \| \nabla \phi \|_{L^{\infty}(\Omega)   }}{ \|f \|_{L^{\infty}(\Omega) }}  \right\}  $$
and  
$$  c_p:= \sup_{f \in L^{
p}(\Omega)} 
\left\{\frac{ \|  \phi \|_{L^{\infty}(\Omega)   }}{ \|f \|_{L^{p}(\Omega) }}  \right\}  $$
where  $\phi$ is the solution to  the problem 
$$
\left\{  \begin{array}{ll} 
- \Delta \phi + \phi =f, \quad & x\in \Omega,
\\ [2mm]  \displaystyle 
\frac{\partial \phi }{\partial \overrightarrow{n} }=0, & x\in \partial \Omega
\end{array} \right.
$$
for $f \not\equiv 0$, $f \in L^{\infty} (\Omega)$ or  $f \in L^p( \Omega)$ respectively.  
\end{definition} 
\begin{remark} 
Notice that for any $p >N+1$
 $$\|f\|_{L^{N+1}(\Omega)} \leq  |\Omega|^{ \frac{p-N-1}{p(N+1)} } \|f\|_{L^p(\Omega)} $$
we have that
$$c_p =\sup_{f \in L^{
p}(\Omega)} \left\{ \frac{ \|  \phi \|_{L^{\infty}(\Omega)   }}{ \|f \|_{L^{p}(\Omega) }} \right\}
\leq \sup_{f \in L^{N+1}(\Omega)}   \left\{ |\Omega|^{ \frac{p-N-1}{p(N+1)} }  \frac{ \|  \phi \|_{L^{\infty}(\Omega)   }}{ \|f \|_{L^{N+1}(\Omega) }}  \right\}
= |\Omega|^{ \frac{p-N-1}{p(N+1)} } c_{N+1}
$$
and taking limits when $p \rightarrow \infty$ 
$$c_{\Omega} \leq  |\Omega|^{ \frac{1}{N+1} } c_{N+1}.$$
\end{remark} 
   %
   %
   %
   %
   %
   %
   %
   %
   %
   %
   %
   %
   %
   %
   %
   
   \section{Local existence of solutions} \label{s2} 
   \setcounter{equation}{0}
  We first consider the approximated problem 
\begin{equation} \label{2.1} \left\{ \begin{array}{l} 
u_{nt} - \Delta( \gamma(v_n) u_n) = \mu u_n (1- (u_n)_+),  
\\ [2mm] 
-\Delta v_n+  v_n  =  \frac{ u_n}{1+ \frac{(u_n)_+}{n} } , 
\\[2mm] 
\frac{\partial u_n}{\partial \overrightarrow{n}}
=  \frac{\partial v_n}{\partial \overrightarrow{n}}=0,   \quad  x\in \partial \Omega,  \\ [4mm]
u_n(0,x)=u_{n0}(x),  
\end{array} \right.
\end{equation} 
where $( \ \cdot \  )_+$ indicates the positive part function. We work  with weak  solutions of the approximated problem, which are  given in the following definition. 
  \begin{definition}  \label{defun}  We say that
  $(u_n, v_n)$ is a weak solution to (\ref{2.1}) if $$ u_n \in L^2(0,T:H^2(\Omega)) \cap H^1(0,T:L^2(\Omega)) \cap C(0,T:L^2(\Omega) ),$$
    $ v_n \in  C(0,T:H^1(\Omega)) $ and for any $\phi \in L^2(0,T:H^2(\Omega)) \cap H^1(0,T:L^2(\Omega)) \cap C(0,T:L^2(\Omega) ) $ such that
  $$\frac{\partial \phi}{\partial \overrightarrow{n}}  =0 , \quad \mbox{ in }  x\in \partial \Omega $$ 
  we have that 
  $$\begin{array}{lcl} \displaystyle  -\int_0^T\int_{\Omega} u_n \phi_t dx dt -\int_0^T \int_{\Omega} u_n \gamma(v_n) \Delta \phi dxdt & = &\displaystyle \mu \int_0^T \int_{\Omega} u_n(1-(u_n)_+) \phi dxdt
  \\ [4mm] &
  +& \displaystyle  \int_{\Omega}[ u_{n0} \phi_0-u(T) \phi(T)]dx
  \end{array} 
  $$   
  and for any $t \in (0,T)$ 
  $$ \int_{\Omega} \nabla v_n \nabla \phi dx+ \int_{\Omega}  v_n   \phi dx=   \int_{\Omega} 	\frac{u_n}{1+ \frac{(u_n)_+}{n}}  \phi dx. 
  $$   
  \end{definition} 
Now we introduce the following ODE problem for $p\geq \max\{4, N+1\} $
\begin{equation} 
\label{ode1} 
\frac{1}{p} y^{\prime} = \mu |\Omega| + c_{\Omega} c_{\gamma} 
y^{\frac{p+2}{p}}.
\end{equation} 
   \begin{lemma}\label{l2.1} Let $p>\max\{4, N\} $ and  $y_0 \geq 0$, then, there exists $T_p>0$ and a  unique solution to (\ref{ode1}) satisfying    $y(0)=y_0$  in $(0,T_p)$.
\end{lemma}
\begin{proof} 
Existence of solutions is a consequence of Peano$^{\prime}$s Theorem.  Uniqueness  is deduced from the fact that   the right-hand side term is locally Lipschitz.  
\end{proof} 
In the following lemmas we obtain  some  a priori estimates to finally prove the local existence of solutions  
   \begin{lemma}\label{l2.2} Let $u_n$ be   a weak solution to (\ref{2.1}) in the sense of Definition \ref{defun}, then,   we have that 
  $$u_n  \geq 0.$$
\end{lemma}
\begin{proof} 
Since $$  \frac{ u_n}{1+ \frac{(u_n)_+}{n} }  \in L^{\infty}( \Omega),$$ 
we have that $v_n \in W^{1, \infty}( \Omega)$.  
 We multiply by $\phi(u_n):= -(-u_n)_+$ and after integration by parts we obtain 
\begin{equation}
 \label{2.1b} \begin{array}{lcl} \displaystyle 
\frac{d}{dt} \frac{1}{2}  \int_{\Omega} [ \phi(u_n)]^2 dx+  \int_{\Omega} \gamma(v) \nabla u_n  \nabla \phi (u_n) dx & + &  \displaystyle  \int_{\Omega} \gamma^{\prime} (v_n) u_n \nabla v_n  \nabla \phi (u_n) dx
\\ [4mm] &  = & \displaystyle 
 \mu    \int_{\Omega} u_n \phi(u_n)  (1-(u_n)_+) dx.
 \end{array}
\end{equation} 
Since 
    $$  \begin{array}{rcl} \displaystyle 
      \int_{\Omega} \gamma(v) \nabla u_n  \nabla \phi (u_n) dx & = & \displaystyle  \int_{\Omega} \gamma(v) |  \nabla \phi (u_n)|^2 dx;
      \\ [4mm]   \displaystyle 
     \int_{\Omega} \gamma^{\prime}(v) u \nabla v  
   \nabla \phi(u_n) dx& \leq & \displaystyle \epsilon  \int_{\Omega} \gamma(v) | \nabla \phi (u_n) |^2dx
     \\ [4mm] &+&   \displaystyle  c(\epsilon)  \int_{\Omega} \frac{| \gamma^{\prime}(v_n)|^2}{\gamma(v_n)}  |\nabla v_n|^2 |\phi(u_n)|^2  dx;
\\ [4mm]   \displaystyle 
 \mu    \int_{\Omega} u_n \phi(u_n)  (1-(u_n)_+) dx & \leq & \displaystyle \mu  \int_{\Omega} \phi(u_n)^2dx,     \end{array} 
$$
and  in view of $v_{n} \in L^{\infty}( \Omega)$ and thanks to assumption (\ref{H4}) 
we get 
$$ 
\frac{|\gamma^{\prime} (v_n) |^2}{\vert \gamma(v_n) \vert}  \leq c_{\gamma} .$$ Then,  (\ref{2.1b}) becomes 
$$
\frac{d}{dt} \frac{1}{2}  \int_{\Omega} [ \phi(u_n)]^2 dx+ (1-\epsilon)  \int_{\Omega} \gamma(v) | \nabla \phi (u_n) |^2dx\leq c
    \int_{\Omega} | \phi(u_n) |^2 dx
$$
and Gronwall$^{\prime}$s Lemma ends the proof. 

\end{proof}

   \begin{lemma} \label{l2.3} Let $u_n$ be   a weak solution to (\ref{2.1}) in the sense of Definition \ref{defun}, then,   we have that 
   $$\int_{\Omega} |u_n| dx \leq c.$$
   \end{lemma}
   \begin{proof}
   We integrate over $\Omega$ to obtain 
   $$ \frac{d}{dt} \int_{\Omega} u_n dx= \mu \int_{\Omega}  u_n dx - \mu  \int_{\Omega}  u_n^2 dx$$ 
   Thanks to Cauchy-Schwarz inequality we get 
   $$ \int_{\Omega}  |u_n |dx \leq |\Omega|^{\frac{1}{2}}
    \left[ \int_{\Omega} u_n^2 dx\right]^{\frac{1}{2} }   $$ 
   which implies  $$ \frac{d}{dt} \int_{\Omega} u_n dx+ \frac{ \mu}{|\Omega|}    \left[ \int_{\Omega} u_n dx\right]^{2} 
    \leq  \mu \int_{\Omega}  u_n dx. $$ 
  Growall$^{\prime}$s Lemma ends the proof.  
   \end{proof} 
   
    \begin{lemma}\label{l2.4}  Let $p \geq \max\{4,N+1\}$, then, there exists $T_p>0$ independent of $n$ such that 
   $$\int_{\Omega} |u_n|^{p}  dx \leq c \quad \mbox{ for $t< T_p.$} $$
   \end{lemma}
   \begin{proof}
  We multiply by $u^{p-1} $ for $p>N$ and integrate by parts to obtain 
  
   \begin{equation}
 \label{2.2} \begin{array}{lcll} \displaystyle 
\frac{d}{dt} \frac{1}{p}  \int_{\Omega} [ u_n]^p dx &+  & \displaystyle  (p-1) \int_{\Omega} \gamma(v)  u^{p-2} | \nabla u_n |^{2}  dx &   \\ [4mm] 
\displaystyle & +& \displaystyle (p-1)  \int_{\Omega} \gamma^{\prime} (v_n) u_n^{p-1} \nabla v_n  \nabla u_n dx
&    \displaystyle
=  \mu    \int_{\Omega} u_n^{p}   (1-u_n) dx,
 \end{array} 
\end{equation} 
   since 
     $$  \begin{array}{rcl} \displaystyle 
    (p-1)  \int_{\Omega} \gamma(v) u_n^{p-2} | \nabla u_n|^2 dx & = & \displaystyle       \frac{4(p-1)}{p^2}  \int_{\Omega} \gamma(v) | \nabla u_n^{\frac{p}{2}} |^2 dx
    \\ [4mm]   \displaystyle 
 \mu    \int_{\Omega} u_n^p    (1-u_n) dx & \leq & \displaystyle  \mu |\Omega| 
      \\ [4mm]   \displaystyle 
   (p-1)  \int_{\Omega} \gamma^{\prime} (v_n) u_n^{p-1} \nabla v_n  \nabla u_n dx & \leq & \displaystyle \epsilon  \int_{\Omega} \gamma(v) | \nabla u_n^{\frac{p}{2}} |^2 dx
     \\ [4mm]   & + &  \displaystyle c(\epsilon)  \int_{\Omega} \frac{
   | \gamma^{\prime}(v_n) |^2}{\gamma(v_n)}  |\nabla v_n|^2 u_n^p  dx
     \\ [4mm]   \displaystyle  \int_{\Omega} \frac{| \gamma^{\prime}(v_n)|^2}{\gamma(v_n)}  |\nabla v_n|^2 u_n^p  dx
& \leq & \displaystyle \left\| \frac{| \gamma^{\prime}(v_n)|^2}{\gamma(v_n)} \right\|_{L^{\infty}( \Omega)} 
  \|\nabla v_n\|_{L^{\infty}(\Omega) }^{2}   
\left[ \int_{\Omega}  u_n^{p}  dx\right]
     \\ [4mm]  
& \leq & \displaystyle \left\| \frac{| \gamma^{\prime}(v_n)|^2}{\gamma(v_n)} \right\|_{L^{\infty}( \Omega)} 
c_{p}^2
\left[ \int_{\Omega}  u_n^{p}  dx\right]^{\frac{p+2}{p}}
     \end{array} 
$$
where $c_{p}$ has been defined in (\ref{dc}).
   Then, (\ref{2.2}) becomes 
$$
\frac{d}{dt} \frac{1}{p}  \int_{\Omega} [ u_n]^p dx \leq \mu |\Omega|  + c^2_{p}    \left\| \frac{|\gamma^{\prime} (v_n) |^2}{\gamma(v_n)} \right\|_{L^{\infty}(\Omega)}
\left[ \int_{\Omega}  u_n^{p}  dx\right]^{\frac{p+2}{p}}
$$
Thanks to assumption (\ref{H4}) it becomes 
$$
\frac{d}{dt} \frac{1}{p}  \int_{\Omega} [ u_n]^p dx \leq \mu |\Omega|  + c^2_{p}    c_{\gamma} \left[ \int_{\Omega}  u_n^{p}  dx\right]^{\frac{p+2}{p}}
$$
In view of Lemma \ref{l2.1}, standard comparison methods prove  the existence of a positive  $T_{p}>0$ such that $ \int_{\Omega} [u_n]^p dx $ is bounded for any $t<T_p$.  
   \end{proof}
     \begin{lemma} \label{l2.5}
Let  $T_*:=T_{\max\{ 4,N+1\} }$ for $T_{\max\{ 4,N+1\} } $  defined in Lemma \ref{l2.1}, then, for any $t<T_*$  we have that 
  $$\|\nabla v_n \|_{L^{\infty}}( \Omega)\leq c$$
   for any $t<T_*$.
   \end{lemma}
   \begin{proof}
Thanks to Lemma \ref{l2.3} and Theorem 8.31 in  Gilbard and Trudinger \cite{gt}, we have that 
$$ v_n \in W^{2,q}( \Omega)$$
for any $q>N$, the Sobolev embedding  $W^{2,q} (\Omega) \hookrightarrow W^{1, \infty}( \Omega)$  proves the result. 
\end{proof} 
   \begin{lemma} \label{l2.6}
Let $p\geq \max\{4,N+1\} $ and  $T_{*}$ be  defined in Lemma \ref{l2.5},    then, for any $t<T_{*}$  we have that 
  $$\| u_n \|_{L^{\infty}( \Omega)}\leq c.$$
   \end{lemma}
   \begin{proof}
 As in Lemma \ref{l2.4} we  take  $u^{p-1}$ as test function in the weak formulation of problema (\ref{2.1}) to get, after some computations 
$$
  \begin{array}{lcl} \displaystyle 
\frac{d}{dt} \frac{1}{p}  \int_{\Omega} [ u_n]^p dx &\leq & \displaystyle \frac{(p-1)  c^2_{\gamma}}{2} 
  \int_{\Omega} |\nabla v_n  |^2u_n^p dx+c
  \\ [4mm] 
&  \leq &  
 \displaystyle   \frac{(p-1) c^2_{\gamma}}{2} 
c_{N+1}^2  \left[ \int_{\Omega}  u_n^{N+1}  dx \right]^{\frac{2}{N+1}}
  \int_{\Omega}  u_n^{p}  dx  +c.
     \end{array}
$$
Since $ \int_{\Omega}  u_n^{N+1}  dx$ is bounded for any $t<T_{*}$ we have thanks to Gronwall$^{\prime}$s lemma that 
$$  \frac{1}{p}\int_{\Omega} [ u_n]^p dx \leq  exp\{ \int_{0}^t   \frac{(p-1) c^2_{\gamma}}{2} 
c_{N+1}^2  \left[ \int_{\Omega}  u_n^{N+1}  dx \right]^{\frac{2}{N+1}} \}   \frac{1}{p}\int_{\Omega} [ u_0]^p dx+ c
$$
which proves, after taking p-roots in the previous inequality that 
$$ \| u_n \|_{L^p(\Omega)}  \leq    exp\{ \int_{0}^t   \frac{ c^2_{\gamma}}{2} 
c_{N+1}^2  \left[ \int_{\Omega}  u_n^{N+1}  dx \right]^{\frac{2}{N+1}} \} c $$
for some $c$ independent of $p$. 
We take limits when $p $ goes to $+ \infty$ for any $t<T_{*}$ to end the proof.  
\end{proof}

    \begin{lemma} \label{l2.7}
Let  $T_{*}$ be defined in Lemma \ref{l2.5} then, for any $T<T_*$  we have that 
  $$\int_{\Omega} u_n^2(T)dx +  \int_0^T\int_{\Omega} |\nabla u_n|^2dx dt \leq  c(T)$$
  and 
  $$ \int_0^T\int_{\Omega} |u_{nt}|^2  dx \leq c(T).$$
   \end{lemma}
   \begin{proof}
We take $u_n$ as test function in the weak formulation of (\ref{defun}) to obtain 
$$
\begin{array}{lcl} \displaystyle  \int_{\Omega} u_n^2(T)dx + c_0 \int_0^T\int_{\Omega} |\nabla u_n|^2dx dt &  \leq & 
\displaystyle 
 \int_0^T\int_{\Omega} \gamma^{\prime}(v_n) u_n \nabla u_n \nabla v_n dxdt 
 \\ [4mm] 
 &+& \displaystyle  \mu \int_0^T \int_{\Omega}   
u_n^2(1-u_n) dxdt +c.
\end{array} 
$$
Since 
$$\begin{array}{lcl} \displaystyle  \int_0^T\int_{\Omega} \gamma^{\prime}(v_n) u_n \nabla u_n \nabla v_n dxdt
& \leq & \epsilon  \int_0^T\int_{\Omega} |\nabla u_n|^2dx dt
\\ [4mm] 
&+ & \displaystyle  \int_0^T\int_{\Omega} 
|\gamma^{\prime} (v_n)|^2u_n^2 |\nabla v_n|^2dx dt  
\end{array} $$
we have that 
$$\int_{\Omega} u_n^2dx +  \int_0^T\int_{\Omega} |\nabla u_n|^2dx dt \leq  c(T),$$
which proves the first part of the theorem. To prove the second part we  take  $\Delta u_{n}$ in the weak formulation to get  
 $$\begin{array}{rcl} \displaystyle  \left.\int_{\Omega} |\nabla u_{n}|^2dx \right \vert_0^T & + & 
  \displaystyle   \int_0^T\int_{\Omega} \gamma(v_n) |\Delta  u_n|^2dxdt  \leq  \int_0^T\int_{\Omega} u_n|\gamma^{\prime   }(v_n)| \Delta v_n| | \Delta u_n |
 dx dt \\ [4mm] &+& \displaystyle 
  \int_0^T\int_{\Omega} \left[ |\gamma^{\prime   }(v_n)| | \nabla v_n| | \nabla u_n |+ \gamma^{\prime \prime   }(v_n)| \nabla  v_n| | \nabla u_n  | \right] |\Delta u_n |
dxdt 
\\ [4mm] &+&  \displaystyle  \mu \int_{0}^T\int_{\Omega}   
(u_n+u_n^2) |\Delta u_n|dxdt  . \end{array}$$
Since $$\int_{\Omega} |\Delta v_n|^p dx \leq c(T)$$
we get 
 $$ \left.\int_{\Omega} |\nabla u_{n}|^2dx \right\vert_0^T+
    \int_0^T\int_{\Omega}   |\Delta  u_n|^2dxdt  \leq  c(T) .$$
We multiply now by $u_{nt}$ and, in view of 
$$\begin{array}{lcl}  \displaystyle \int_{\Omega} u_t \Delta \gamma(v) u dx  & \leq &\displaystyle \epsilon  \int_{\Omega} |u_{nt} |^2dx
\\ [4mm] &+ & \displaystyle  c \int_{\Omega} |\Delta u_{nt} |^2dx + c \int_{\Omega} |\Delta v_{nt} |^2dx
+c \int_{\Omega} |\nabla u_{nt} |^2 |\nabla v_n|^2dx \end{array} $$
and previous lemmas we get that 
$$ \int_0^T\int_{\Omega} |u_{nt}|^2  dx \leq c$$
which ends the proof. 
\end{proof} 
\begin{lemma} \label{l2.8} Let $T<T_*$ small enough, then, 
   there exists a unique weak solution $(u_n,v_n)$ to (\ref{2.1}) in the sense of definition \ref{defun} such that 
   $$ u_n \in  L^2(0,T:H^2(\Omega)) \cap L^{\infty}(0,T: H^1(\Omega)), \qquad v\in C^0( 0,T: C^1(\overline{\Omega})) .$$
\end{lemma} 
\begin{proof} We consider a fixed point argument, for a given $\tilde{u}_n\in C^{ \frac{\beta}{2}, \beta}_{t,x}([0,T] \times \overline{\Omega})$, $\tilde{u}_n \geq 0$,  
we consider the solution  $v_n$ to the problem 
\begin{equation} \label{vn}- \Delta v_n+v_n=\frac{\tilde{u}_n}{1+ \frac{\tilde{u}_n}{n}} \end{equation} 
with Neumann boundary condition. 
Then, we consider $u_n$ the solution to the parabolic problem 
\begin{equation} \label{unfijo} \left\{ \begin{array}{l} u_{nt}- div( \gamma(v_n) \nabla u_n)= -div ( u_n \gamma^{\prime}(v_n)\nabla v_n) + \mu u_n(1- \tilde{u}_n),
\\ [4mm]  u(0,x)=u_0(x), 
\end{array} \right.
\end{equation} 
with Neumann boundary conditions.   Lax - Milgram theorem proves the existence of a unique solution $v_n$ to  (\ref{vn}),  
thanks to Gilbart and Trudinger \cite{gt} Theorem 8.34 we have that
 $v_n \in C^{ \frac{\beta}{2}, 1+ \beta}_{t,x} ([0,T] \times \overline{\Omega})$. 
 We replace $v_n$ into (\ref{unfijo}) to obtain thanks to Lieberman \cite{lieberman}, Theorem 4.30, page 79 that   $u_n$ satisfies
 $u_n  \in C^{ 1+ \frac{\beta}{2}, 2+ \beta}_{t,x} ([0,T] \times \overline{\Omega})$. 
 Schauder fixed point Theorem proves the existence of solutions for $T$ small enough. Since, for any $n$,   the term $ \frac{\tilde{u}_n}{1+\frac{\tilde{u}_n}{n}}$ is bounded in $L^{\infty}[(0,T_{*}) \times  \Omega)$ we   extend the solution up to $T=T_{*}$.   
 Uniqueness is obtained by standard arguments,  we assume the existence of    two different solutions to get a contradiction.  
 The proof is similar to the proof given in   \cite{tw1}  Theorem 2.1, therefore we omit the details. 
\end{proof} 

   \begin{theorem} \label{theorem2.1} Let $T<T_*$, then, 
   there exists a unique weak solution $(u,v)$ to (\ref{1.1}) in the sense of definition \ref{defu} such that 
   $$ u \in  L^2(0,T:H^2(\Omega)) \cap L^{\infty}(0,T: H^1(\Omega)), \qquad v\in C^0( 0,T: C^1(\overline{\Omega})) .$$
   \end{theorem} 
   \begin{proof}
   We have that for any $T<T_*$ 
    $$\int_0^T\int_{\Omega} |u_{nt}  |^2 dx \leq c(T),  $$
     $$\int_0^T\int_{\Omega} |\Delta u_{n}  |^2 dx \leq c(T),  $$
    and 
      $$\| \nabla u_n \|_{[ L^{2}(\Omega)]^N}  \leq c(T).$$ 
      Then, $v_n$ is  bounded in $L^{\infty}(0,T: W^{2,p}(\Omega))$. Moreovere if  
      we  derivate respect to $t$ the equation of $v_n$ it results 
      $$- \Delta v_{nt} + v_{nt} =  \frac{u_{nt}}{(1+ \frac{u_n}{n})^2}$$
      we take squares in the previous inequality to get, after integration 
     $$\int_0^T\int_{\Omega} (|\Delta v_{nt}  |^2  + |u_{nt}  |^2 + |\nabla u_{nt}  |^2) dxdt \leq c(T) . $$
      Thanks to Aubin-Lions Lemma, for the spaces $H^2 (\Omega) \hookrightarrow  H^1(\Omega) \hookrightarrow L^2(\Omega)$, 
      there exists $u\in C(0,T:L^2(\Omega))$ such that 
      $$u_n\rightarrow u \quad \mbox{ strong in } L^2(0,T:H^1(\Omega))$$
        $$u_n\rightharpoonup u \quad \mbox{ weak in } L^2(0,T:H^2(\Omega))$$
      and since $W^{2,p} (\Omega) \hookrightarrow  C^1(\overline{\Omega}) \hookrightarrow L^2(\Omega),$ 
 there exists $v\in C(0,T:C^1(\overline{\Omega}))$ satisfying 
      $$v_n\rightarrow v \quad \mbox{ strong in } C^0(0,T:C^1(\overline{\Omega})).$$
      Then  
    $$\gamma(v_n ) \rightarrow \gamma(v), \quad \mbox{ strong in }  C^0( 0,T: C^1(\overline{\Omega})) $$
    and 
$$\gamma^{\prime}(v_n ) \rightarrow \gamma^{\prime}(v), \quad \mbox{ strong in }  C^0( 0,T: C^1(\overline{\Omega}) )$$
 and therefore we get the following convergence of 
 the integrals 
  $$\int_0^T\int_{\Omega} u_n \phi_t dx dt 
  \rightarrow \int_0^T\int_{\Omega} u \phi_t dx dt $$
  $$
  \int_0^T \int_{\Omega}   \gamma(v_n) \nabla u_n \nabla  \phi dxdt \rightarrow  \int_0^T \int_{\Omega}   \gamma(v_n) \nabla u_n \nabla  \phi dxdt 
  $$
  $$
   \mu \int_0^T \int_{\Omega} u_n(1-(u_n)_+) \phi dxdt \rightarrow    \mu \int_0^T \int_{\Omega} u_n(1-(u_n)_+) \phi dxdt $$
$$ \int_{\Omega} u_n(T) \phi(T)]dx \rightarrow  \int_{\Omega} u(T) \phi(T)]dx$$
  $$ \int_{\Omega} \nabla v_n \nabla \phi dx \rightarrow  \int_{\Omega} \nabla v_n \nabla \phi dx,
  $$   
   $$  \int_{\Omega}  v_n   \phi dx \rightarrow   \int_{\Omega}  v   \phi dx ,
  $$  
     $$    \int_{\Omega} 	\frac{u_n}{1+ \frac{(u_n)_+}{n}}  \phi dx \rightarrow    \int_{\Omega} 	u_n \phi dx .
  $$ 
  We take limits in (\ref{defun}) to obtain that $(u,v)$ is a weak solution of (\ref{1.1}) in the sense of definition \ref{defu}.  Since $\nabla v $ is bounded in $[0,T] \times \Omega$ for any $T<T_*$ we have that standard  parabolic regularity shows 
  $u \in L^q(0,T:W^{2,q}(\Omega)) \cap W^{1,q}(0,T:L^q(\Omega))$
  for any $q<\infty$.  
  
  To obtain uniqueness we proceed by contradiction   assuming  there are two different solutions. Standard computations   show uniqueness of solutions. Since the  proof is similar to 
  the proof of  Theorem 2.1 in Tello and Winkler \cite{tw1} therefore we omit the details. 

    \end{proof}

   \section{Auxiliary  problem } \label{s3} \setcounter{equation}{0}
   We first notice that the term $\Delta [ \gamma(v)u]$ in (\ref{1.1}) is  expressed  as follows 
$$\begin{array}{lcl} \displaystyle  -\Delta (u \gamma(v)) & = & \displaystyle -\gamma(v) \Delta u - 2\gamma^{\prime}(v) \nabla v \nabla u- u \gamma^{\prime } (v) \Delta v -u \gamma^{\prime \prime} (v)|\nabla  v|^2$$
$$\\ [2mm] & = & \displaystyle   -\gamma(v) \Delta u -2 \gamma^{\prime}(v) \nabla v \nabla u+ u \gamma^{\prime } (v)(u-v) -u \gamma^{\prime \prime} (v)|\nabla  v|^2, 
\end{array} 
$$
we replace the previous  expression into (\ref{1.1}) to obtain
$$ u_t -\gamma(v) \Delta u= 2\gamma^{\prime}(v) \nabla v \nabla u-u \gamma^{\prime } (v)(u-v)+u \gamma^{\prime \prime} (v)|\nabla  v|^2+  \mu u(1-u). 
$$
Let $a(t)$ defined by 
$$a(t):= \|\nabla v\|_{L^{\infty}(\Omega)}$$ 
and consider  the system of Ordinary Differential Equations in the time interval $[0,T]$ for some $T\in (0,T_*)$, 
\begin{equation} 
\label{3.0} \left\{ \begin{array}{rcl} \overline{u}_t & = &  \overline{u} [-\gamma^{\prime } (\underline{u})] (\overline{u} -\underline{u})+\overline{u} \gamma^{\prime \prime} (\underline{u}) a(t) +  \mu\overline{u}(1-\overline{u}), 
\\ [2mm] \underline{u}_t & = &  \underline{u}[-  \gamma^{\prime } (\underline{u})] (\underline{u} -\overline{u})  +  \mu\underline{u}(1-\underline{u}), 
\\ [2mm]  & &  \overline{u}(0)   =  \overline{u}_0,  \quad  \underline{u}(0)= \underline{u}_0.
\end{array} \right.
\end{equation} 
To simplify the previous system, we divide by $\overline{u}$  the first equation and by  $\underline{u}$ to have 
\begin{equation} \label{ODE} \left\{ \begin{array}{lcl}  \frac{d}{dt} \log( \overline{u})  & = & [- \gamma^{\prime } (\underline{u})] (\overline{u} -\underline{u})+ 
 \gamma^{\prime \prime} (\underline{u}) a(t)  +  \mu (1-\overline{u}),   
\\ 
\frac{d}{dt} \log(\underline{u}) & =  &  [- \gamma^{\prime } (\underline{u})] (\underline{u} -\overline{u})  +  \mu (1-\underline{u}),  
\\ [2mm]  & &  \overline{u}(0)   =  \overline{u}_0,  \quad  \underline{u}(0)= \underline{u}_0.
\end{array} \right. \end{equation}
\begin{lemma}   \label{l3.1} There exists $T_{0} >0$, and $C^1(0,T_{0}) $ functions  $\overline{u}$,    $ \underline{u}$ such that  $(\overline{u}, \underline{u})$  is the unique  solution to (\ref{ODE}) to (\ref{3.0}).  
\end{lemma} 
\begin{proof}
Since $\gamma \in C^3$ and $a(t) \in C^0(0,T)$ (for $T<T_*$),   then, thanks to Peano$^{\prime}$s theorem   there exists a local solution to (\ref{ODE}). Moreover since the righthand-side part of the system is locally Lipschitz in $\overline{u}$ and $\underline{u}$  we deduce the uniqueness of solutions. We also may extend the solution to a maximal interval of existence given  by  $(0,T_{0})$ for some positive $T_{0}\leq T_*$. Since $a(t)$ is a continuous function in $[0,T_*)$ (for $T_*$ defined in 
Lemma\ref{l2.6})
 we obtain the wished regularity and conclude   the proof. 

\end{proof}
\begin{lemma} \label{l3.2}  Let $\overline{u}$ and $ \underline{u}$ the solutions to (\ref{ODE}) in $(0, T_{0})$, such that  
 $$ 0  < \underline{u}_0 < 1< \overline{u}_0, \quad \mbox{ for any }  t<T_{0} $$ then  
\begin{equation} \label{t0} 0  < \underline{u} < 1< \overline{u}. \end{equation} 
\end{lemma} 
\begin{proof} We argue by contradiction and assume that
\begin{equation} \label{cont} \mbox{  \emph{there exists $t_0 \in (0,T_{0}) $, such that  (\ref{t0}) is   satisfied for any $t<t_0$}} \end{equation} 
 and    
 \begin{itemize} 
\item[(1)] $\underline{u}(t_0)=1$ and $ \overline{u}(t_0)>1;$
\item[(2)] $\underline{u}(t_0)<1$ and $ \overline{u}(t_0)=1;$
\item[(3)] $\underline{u}(t_0)=1$ and $ \overline{u}(t_0)=1;$
\newline 
or  either
\item[(4)]  $\underline{u}(t_0)=0$ and $ \overline{u}(t_0) \geq 1.$
\end{itemize} 
In case (1), we have that, by substituting in (\ref{ODE}),    $$ \underline{u}^{\prime}(t_0)<0$$ which contradicts   (\ref{cont}) because at $t=t_0$,  $\underline{u}$ gets its minimum of
$(0,t_0]$. 
In the same fashion we see that (2) is not possible. In case (3),   
we subtruct both equations to get 
$$\begin{array}{lcl} \frac{d}{dt} \left( \log( \overline{u}) -\log(\underline{u}) \right) &= & 2[- \gamma^{\prime } (\underline{u})](\overline{u} -\underline{u})+  \gamma^{\prime \prime} (\underline{u}) a(t)
  -  \mu (\overline{u}- \underline{u})
\\ [4mm] &= & [- 2\gamma^{\prime } (\underline{u}) - \mu](\overline{u} -\underline{u})+  \gamma^{\prime \prime} (\underline{u}) a(t) 
\\ [4mm] &= & [- 2\gamma^{\prime } (\underline{u}) - \mu] \xi ( \ln(\overline{u}) -\ln(\underline{u})) +  \gamma^{\prime \prime} (\overline{u}) a(t) 
\end{array} 
$$for some $\xi \in ( \underline{u}, \overline{u})$. After integration over $(0,t)$   we have that 

$$ \begin{array}{lcl} \displaystyle 
e^{-\int_{0}^t[- 2\gamma^{\prime } (\underline{u}) - \mu]\xi d\tau }  \left( \log( \overline{u}) -\log(\underline{u}) \right)
&  = & \displaystyle \log( \overline{u}_0) -\log(\underline{u}_0) 
\\ [2mm] & + & 
\displaystyle  \int_0^te^{-\int_{0}^{\tau}[- 2\gamma^{\prime } (\underline{u}) - \mu] \xi ds } \gamma^{\prime \prime} (\underline{u}) a(\tau) d \tau.
\end{array} 
$$
We take $t=t_0<T_{0}$ to obtain  
$$   \log( \overline{u}(t_0)) -\log(\underline{u}(t_o)) >0$$
which contradicts (3). To see  that (4) is not possible, we just notice that $\underline{u}=0$ in $[0,t_0]$ is the backward solution to (\ref{3.0}), thanks to uniqueness of solutions,  
we have that $\underline{u}(0)=0$ which contradicts that $\underline{u}_0>0$ and  the proof ends.
\end{proof}   
\begin{lemma} 
\label{l3.3} Let $T_0$ be given in Lemma \ref{3-1} and $T_+$ in Lemma \ref{l2.6}  then
$T_{0}=T_*.$ 
\end{lemma} 
\begin{proof} 
We subtract in (\ref{ODE}) the second equation to the first to obtain  
$$ \frac{d}{dt} \left[  \log( \overline{u})-  \log(\underline{u}) \right]  = (2[- \gamma^{\prime }] - \mu \underline{u}) (\overline{u} -\underline{u})+ 
 \gamma^{\prime \prime} (\underline{u}) a(t). $$
Lemma \ref{l3.2} and assumption (\ref{H3}) shows 
$$ \frac{d}{dt} \left[  \log( \overline{u})-  \log(\underline{u}) \right]   \leq  
 \gamma^{\prime \prime} (\underline{u}) a(t). $$ 
 In view of $\underline{u} \in (0,1) $, $\gamma^{\prime \prime }(\underline{u}) \leq \gamma^{\prime \prime} (0) <\infty$ we get, after integration in the previous equation  
  the upper boundedness of $   \overline{u}$  and the lower boundedness of $\underline{u}$ for any $t<T_*$. 

\end{proof}

     \section{Comparison principle ODEs system / PDEs system    } \label{s4} \setcounter{equation}{0}
     In this section we compare the solution of system (\ref{1.1}) to $\underline{u}$,  $\overline{u}$,  the solution to system  (\ref{ode1}). We notice that to prove that $u$ and $v$ are  bounded by $\overline{u}$ is equivalent to obtain  the  non-positivity   of the functions $u- \underline{u}$ and $v- \overline{v}$. In the same way we have to see that $u- \underline{u}$ and $v- \underline{v}$ are non-negative functions. To obtain such result we   introduce the following the functions 
     $$\overline{U}= u- \overline{u}, \qquad \underline{U}= u- \underline{u},$$
      $$\overline{V}= v- \overline{v}, \qquad \underline{V}= v- \underline{v}.$$
     \begin{lemma} \label{l4.1} For any $t<T_{*}$ we have that  
     $$\overline{U}\leq 0, \qquad \underline{U}\geq 0. $$
     \end{lemma}
     \begin{proof}
     We first consider the differential equations satisfied by $\overline{U}$ and  $\underline{U}$. Since  $u$ fulfills 
          $$u_t -div(\gamma(v) \nabla u )=\gamma^{\prime} (v) \nabla v \nabla u +  u \gamma^{\prime \prime}(v)|\nabla v|^2 + u [-\gamma^{\prime}(v)](u-v)  + \mu u(1-u), 
          $$ we have that  $\overline{U}$ satisfies 
     $$\begin{array}{l} \overline{U}_t -div(\gamma(v) \nabla \overline{U} )  =   \gamma^{\prime} (v) \nabla v \nabla \overline{U} + 
    ( \overline{U}+ \overline{u}) \gamma^{\prime \prime}(v)|\nabla v|^2 - \overline{u}\gamma^{\prime \prime}( \underline{u}) a(t)     \\ [4mm]  
    + (\overline{U}+ \overline{u})  [-\gamma^{\prime}(v)](u-v) - \overline{u} [-\gamma^{\prime } (\underline{u})] (\overline{u} -\underline{u})+  \mu u(1-u)-  \mu\overline{u}(1-\overline{u})
    \\ [4mm] 
    =   \gamma^{\prime} (v) \nabla v \nabla \overline{U} + 
     \overline{U}\gamma^{\prime \prime}(v)|\nabla v|^2 + \overline{u}[  \gamma^{\prime \prime}(v)|\nabla v|^2- \gamma^{\prime \prime}( \underline{u})  a(t) ]    \\ [4mm]  
    + \overline{U}  [-\gamma^{\prime}(v)](u-v) +  \overline{u} \left\{  [-\gamma^{\prime}(v)](u-v)-   [-\gamma^{\prime } (\underline{u})] (\overline{u} -\underline{u}) \right\}+ 
     \mu \overline{U} -  \mu \overline{U}  (u+\overline{u})   
     \\ [4mm] 
    =   \gamma^{\prime} (v) \nabla v \nabla \overline{U} + 
     \overline{U}\gamma^{\prime \prime}(v)|\nabla v|^2 + \overline{u}[  \gamma^{\prime \prime}(v)(|\nabla v|^2-a(t))- [ \gamma^{\prime \prime}( \underline{u}) - \gamma^{\prime \prime} (v)]  a(t) ]    \\ [4mm]  
    + \overline{U}  [-\gamma^{\prime}(v)](u-v) +  \overline{u} \left\{  [-\gamma^{\prime}(v)](u-v)-   [-\gamma^{\prime } (\underline{u})] (\overline{u} -\underline{u}) \right\}+ 
     \mu \overline{U} -  \mu \overline{U}  (u+\overline{u})   .
\end{array}
$$
In view of definition of $a(t)$, 
we have that $$|\nabla v|^2-a(t)  \leq 0 \quad a.e.$$
%
 and  since 
 $$\overline{u}   [-\gamma^{\prime}(v)](u-v)= \overline{u}   [-\gamma^{\prime}(v)](\overline{U} - \underline{V})+\overline{u}   [-\gamma^{\prime}(v)](\overline{u}-\underline{u})
$$
it results 
  $$\begin{array}{lcl} \displaystyle  \overline{U}_t -div(\gamma(v) \nabla \overline{U} ) & \leq & \displaystyle  
    \gamma^{\prime} (v) \nabla v \nabla \overline{U} + 
     \overline{U}\gamma^{\prime \prime}(v)|\nabla v|^2
      - \overline{u}[  \gamma^{\prime \prime}( \underline{u}) - \gamma^{\prime \prime} (v)]  a(t)     \\ [4mm]  &
    + & \displaystyle  \overline{U}  [-\gamma^{\prime}(v)](u-v)  
    \\ [4mm]  & + & \displaystyle   \overline{u} \left\{ [-\gamma^{\prime}(v)](\overline{U} - \underline{V}) 
    ([-\gamma^{\prime}(v)]- [-\gamma^{\prime } (\underline{u})] )(\overline{u}-\underline{u})\right\}    \\ [4mm]  & + & \displaystyle    \mu \overline{U} -  \mu \overline{U}  (u+\overline{u}) .  
     \end{array}
     $$
     Since 
     $$  \gamma^{\prime \prime}( \underline{u}) - \gamma^{\prime \prime} (v) = -\gamma^{\prime \prime \prime}(\xi_1) \underline{V},$$
     $$ \gamma^{\prime  }( \underline{u}) - \gamma^{\prime } (v) = -\gamma^{\prime  \prime}(\xi_2) \underline{V}$$
     and 
     $$ 0 \leq   \overline{u}- \underline{u}       $$
  we have that 
      $$\begin{array}{l} \overline{U}_t -div(\gamma(v) \nabla \overline{U} )  \leq 
    \gamma^{\prime} (v) \nabla v \nabla \overline{U} + 
     \overline{U}\gamma^{\prime \prime}(v)|\nabla v|^2 - \overline{u} [- \gamma^{\prime \prime \prime}( \xi_1) ]   a(t) \underline{V}    \\ [4mm]  
    + \overline{U}  [-\gamma^{\prime}(v)](u-v) +  \overline{u} \left\{ [-\gamma^{\prime}(v)](\overline{U} - \underline{V})\right\}
    - \overline{u}\gamma^{\prime  \prime}(\xi_2) \underline{V}( \overline{u}- \underline{u}) +
     \mu \overline{U} -  \mu \overline{U}  (u+\overline{u})  . 
     \end{array}
     $$
We now multiply the previous equation by $\overline{U}_+$ and integrate by parts over $\Omega$ to obtain after some computations 
 $$ \begin{array}{ll} \displaystyle 
 \frac{d}{dt} \frac{1}{2} \int_{\Omega} | \overline{U}_+ |^2 dx+ \epsilon \int_{\Omega} | \nabla \overline{U}_+ |^2 dx & \displaystyle  \leq  k  \int_{\Omega} (1+u)| \overline{U}_+ |^2 dx
 \\ 
 [2mm] & \displaystyle -\int_{\Omega} \overline{u} \left\{ -\gamma^{\prime}(v)+ \gamma^{\prime \prime} (\xi_2)  - \gamma^{\prime \prime \prime}( \xi_1) a(t)\right\}  \overline{U}_+ \underline{V} dx \end{array}
 $$
 thanks to assumption (\ref{H2}) we have that 
 $$ \overline{u} \left\{ -\gamma^{\prime}(v)+ \gamma^{\prime \prime} (\xi_2)  - \gamma^{\prime \prime \prime}( \xi_1) a(t)\right\}  \geq 0$$
 therefore
$$  -\int_{\Omega} \overline{u} \left\{ -\gamma^{\prime}(v)+ \gamma^{\prime \prime} (\xi_2)  - \gamma^{\prime \prime \prime}( \xi_1) a(t)\right\}    \overline{U}_+ \underline{V} dx 
\leq   \int_{\Omega} | \overline{U}_+ |^2dx +k(t)  \int_{\Omega} | (-\underline{V})_+ |^2 dx.
$$
We notice that 
$$ \int_{\Omega} u| \overline{U}_+ |^2 dx\leq  \|u\|_{L^p(\Omega)}  \|    \overline{U}_+ \|^{2}_{L^{2p^{\prime}}(\Omega)}  $$
for some  $p>N$ we have that $2p^{\prime} = \frac{2p}{p-1} <\frac{2N}{N-2}$ therefore thanks to Gagliardo Nirenberg inequality we have that 
$$  \|    \overline{U}_+ \|^{2}_{L^{2p^{\prime}}(\Omega)} \leq c \|  \overline{U}_+\|_{H^1(\Omega)}^{2a}  \|  \overline{U}_+\|_{L^2(\Omega)}^{2(1-a)}+
c \|  \overline{U}_+\|_{L^2(\Omega)}^{2}
$$ for $a$ satisfying 
$$ \frac{1}{2p^{\prime} }= \left( \frac{1}{2}- \frac{1}{N} \right)a+ \frac{(1-a)}{2}$$
  i.e. 
  $$ a=  \frac{N (p^{\prime} -1)}{2p^{\prime} } = \frac{N  }{2p  } <\frac{1}{2}. $$
 Thanks to  Young inequality and Lemma \ref{l2.4} we 
get 
\begin{equation} 
 \frac{d}{dt} \frac{1}{2} \int_{\Omega} | \overline{U}_+ |^2 dx \leq  k  \int_{\Omega} | \overline{U}_+ |^2 dx+k  \int_{\Omega} | (-\underline{V})_+ |^2 dx
 \label{lg1}
 \end{equation}
 In the same fashion, in view of $u \gamma^{\prime \prime}(v) |\nabla v|^2  \geq 0$, the following inequality is satisfied by $\underline{U} $  
    $$\begin{array}{lcl} \underline{U}_t -div(\gamma(v) \nabla \underline{U} )  & \geq & 
    \gamma^{\prime} (v) \nabla v \nabla \underline{U} 
    + \underline{U}  [-\gamma^{\prime}(v)](u-v) +  \underline{u} \left\{ [-\gamma^{\prime}(v)](u-v)\right\}
     \\ [4mm] & &  -  \underline{u} [-\gamma^{\prime}(\underline{u})](\underline{u}-\overline{u})+
     \mu \underline{U} -  \mu \underline{U}  (u+\underline{u})  .
     \end{array}
     $$
 Since 
 $$ \underline{u} \left\{ [-\gamma^{\prime}(v)](u-v)\right\}= \underline{u} \left\{ [-\gamma^{\prime}(v)](\underline{U} -\overline{V})\right\} 
 + \underline{u} \left\{ [-\gamma^{\prime}(v)](\underline{u} -\overline{u})\right\} 
 $$ we have that 
 $$\begin{array}{lcl} \underline{u} \left\{ [-\gamma^{\prime}(v)](u-v)\right\} & - &   \underline{u} [-\gamma^{\prime}(\underline{u})](\underline{u}-\overline{u})
 \\ [4mm] 
 & = & \underline{u} \left\{ [-\gamma^{\prime}(v)](\underline{U} -\overline{V})\right\} 
 - \underline{u} (\underline{u} -\overline{u}) \left\{ [ \gamma^{\prime}(v)- \gamma^{\prime}(\underline{u}) ]\right\} 
   \\ [4mm] 
 & = &  \underline{u} \left\{ [-\gamma^{\prime}(v)](\underline{U} -\overline{V})\right\} 
- \underline{u} (\underline{u} -\overline{u})   \left[ \gamma^{\prime \prime}(\xi_3)  \underline{V}  \right] 
 \end{array} 
 $$for some $\xi \in (\underline{u},v) $ if $\underline{u}<v$ and $\xi \in [v, \underline{u}]$ otherwise. Then, it results
    $$\begin{array}{lcl} \underline{U}_t -div(\gamma(v) \nabla \underline{U} )  & \geq & 
    \gamma^{\prime} (v) \nabla v \nabla \underline{U} 
    + \underline{U}  [-\gamma^{\prime}(v)](u-v) +  \underline{u} \left\{ [-\gamma^{\prime}(v)](\underline{U} -\overline{V})\right\} 
     \\ [4mm] & &  - \underline{u} (\underline{u} -\overline{u})   \left[ \gamma^{\prime \prime}(\xi_3)  \underline{V}  \right] +
     \mu \underline{U} -  \mu \underline{U}  (u+\underline{u})  .
     \end{array}
     $$
 We  take 
 $-[-\underline{U}]_+$ as test function in the weak formulation to get, after   some computations 
   $$
 \frac{d}{dt} \frac{1}{2} \int_{\Omega} [- \underline{U}]_+ ^2 dx \leq  k  \int_{\Omega}   [- \underline{U}]_+ ^2 dx
 +\int_{\Omega}  \underline{u}  [-\gamma^{\prime}(v)] [-\underline{U}]_+ \overline{V}  dx 
 $$ $$+
 \int_{\Omega}  \underline{u} (\underline{u} -\overline{u})   \gamma^{\prime \prime}(\xi_3)  \underline{V} [- \underline{U}]_+
dx $$
where the term $ \int_{\Omega}  u [- \underline{U}]_+ ^2dx$ is treated as before, using Gagliardo Nirenberg inequality. 
Since 
$$ \int_{\Omega}  \underline{u}  [-\gamma^{\prime}(v)] [-\underline{U}]_+ \overline{V}  dx
\leq \int_{\Omega}  \underline{u}  [-\gamma^{\prime}(v)] [-\underline{U}]_+ \overline{V}_+  dx
 \leq c \int_{\Omega}   [- \underline{U}]_+ ^2 dx+c \int_{\Omega}   \overline{V}_+ ^2 dx
 $$
 and 
 $$   \begin{array}{lcl} \displaystyle   \int_{\Omega}  \underline{u} (\underline{u} -\overline{u})   \gamma^{\prime \prime}(\xi_3)  \underline{V} [- \underline{U}]_+dx
& = & \displaystyle \int_{\Omega}  \underline{u} (\overline{u} -\underline{u})   \gamma^{\prime \prime}(\xi_3) [- \underline{V}]_+ [- \underline{U}]_+dx
\\ [4mm] 
&
\leq & \displaystyle  c \int_{\Omega}   [- \underline{U}]_+ ^2dx +c \int_{\Omega}   [- \underline{V}]_+ ^2 dx
\end{array} 
 $$
it results 
 \begin{equation} 
 \frac{d}{dt} \frac{1}{2} \int_{\Omega} [ -\underline{U}]_+^2 dx \leq  c  \int_{\Omega}   \underline{U}_+^2 dx+c \int_{\Omega}  \overline{V}_+ ^2 dx
 +c\int_{\Omega} [-\underline{V}]_+ ^2 dx. 
 \label{lg2}
 \end{equation}
 In the same way we have that 
\begin{equation} 
\label{V} 
 -\Delta  \underline{V} + \underline{V} = \underline{U} 
 \end{equation} 
 we multiply by $-[- \underline{V}]_+$ and integrate by parts to get, after some computations and thanks to Young Inequality 
\begin{equation}   \int_{\Omega} [-\underline{V}]_+ ^2 dx \leq  \int_{\Omega} [-\underline{U}]_+^2 dx. \label{gl3} \end{equation}
 In the same fashion we obtain 
 \begin{equation}   \int_{\Omega}  \overline{V}_+^2 dx \leq  \int_{\Omega} \overline{U}_+^2 dx. \label{lg4} \end{equation}
Thanks to (\ref{lg1})-(\ref{lg4}) we have 
$$
\frac{d}{dt} \left(   \int_{\Omega}  \overline{U}_+^2 dx + \int_{\Omega} [ -\underline{U}]_+ ^2 dx \right)
\leq c \left(   \int_{\Omega}  \overline{U}_+^2 dx + \int_{\Omega} [ -\underline{U}]_+^2 dx \right).$$
Gronwalls Lemma ends the proof. 
     \end{proof} 
     \section{Asymptotic behaviour}  \label{s5} 
     Thanks to Lemma \ref{l4.1} we have that 
     \begin{equation} 
     \underline{u} \leq u \leq \overline{u}. \end{equation}
     
   \begin{lemma} \label{l5.1} Let $(u,v)$ be the weak solution of (\ref{1.1}) and $(\underline{u}, \overline{u})$ the unique  solution  to (\ref{2.1}), then,  for any $t<T_{*}$ we have that 
   $$ \|\nabla v\|_{L^{\infty} (\Omega)  } \leq c_{\Omega} |\overline{u}- \underline{u}|.$$
   \end{lemma}  
     \begin{proof}
  As a consequence of Lemma  \ref{l4.1} we have that  $u \in L^{\infty}(0,T:L^{\infty}(\Omega))$, for any  $T<T_*$. Then,     from equation (\ref{V})   we deduce 
  $$ \underline{V}  \in W^{2,q}(\Omega)$$  
 for any $t<T_*$ and  $q<\infty$.  
 Therefore, the  constant $c_{\Omega}$ defined in Definition \ref{dc} provided us 
$$\|\nabla \underline{V} \|_{L^{\infty} (\Omega)} \leq c_{\Omega} \|\underline{U}\|_{L^{\infty} (\Omega)} $$
in view of  Lemma \ref{l4.1} we have that 
$$ \|\underline{U}\|_{L^{\infty} (\Omega)} \leq \overline{u} - \underline{u}$$
 which ends the proof. 
         \end{proof} 
    
       \begin{lemma} \label{l5.2}  Let   $(\underline{u}, \overline{u})$ be the unique  solution  to (\ref{2.1}), then,   we have that $T_0= \infty$ and satisfies 
   $$  |\overline{u}- \underline{u}| \rightarrow 0 \qquad \mbox{ as $t \rightarrow \infty$}.$$
   \end{lemma}  
     \begin{proof}
    In view of Lemma \ref{l5.1}, we have that   $ c_{\Omega} |\underline{u}- \overline{u} |^2 $ is an upper bound of $a(t)$, therefore $\overline{u}$ and $\underline{u}$ 
    satisfy   
$$ \left\{ \begin{array}{ll}  \frac{d}{dt} \log( \overline{u}) \leq   [- \gamma^{\prime } (\underline{u})] (\overline{u} -\underline{u})+  \gamma^{\prime \prime} (\underline{u}) c_{\Omega} 
 | \overline{u} - \underline{u}|^2  +  \mu (1-\overline{u}), & t>0 
\\ 
\frac{d}{dt} \log(\underline{u})=   [- \gamma^{\prime } (\underline{u})] (\underline{u} -\overline{u})  +  \mu (1-\underline{u}), & t>0
\\ \overline{u}(0)= \overline{u}_0, \qquad 
\end{array} \right. 
$$
we subtract both equations to get 
$$\begin{array}{lcl} \frac{d}{dt} \left( \log( \overline{u}) -\log(\underline{u}) \right) & \leq &  2[- \gamma^{\prime } (\underline{u})](\overline{u} -\underline{u})+  \gamma^{\prime \prime}(\underline{u}) c_{\Omega}  | \overline{u} - \underline{u}|^2
  -  \mu (\overline{u}- \underline{u})
\\ [2mm] 
& = & [- 2\gamma^{\prime } (\underline{u}) - \mu](\overline{u} -\underline{u})+  \gamma^{\prime \prime} (\underline{u}) c_{\Omega}  | \overline{u} - \underline{u}|^2 
\\ [2mm] 
& \leq &  [- 2\gamma^{\prime } (\underline{u}) - \mu+ \gamma^{\prime \prime} (\overline{u}) c_{\Omega} \underline{u}](\overline{u} -\underline{u}).
\end{array} 
$$
Thanks to assumption (\ref{H3}) it results  
$$\frac{d}{dt} \left( \log( \overline{u}) -\log(\underline{u}) \right) \leq (\mu_0 -\mu)(\overline{u} -\underline{u}).$$
After integration over $(0,t)$, in view of non-negativity of the righthand side term,  it results  
$$   \log( \overline{u}) -\log(\underline{u})  \leq   \log( \overline{u}_0) -\log(\underline{u}_0) .$$
Thanks to assumption (\ref{H5}) we have  
$$\overline{u} \leq \frac{\overline{u}_0}{\underline{u}_0} \underline{u}$$ 
which implies 
\begin{equation} \label{5.1}  0<  \frac{\underline{u}_0}{\overline{u}_0} \leq  \underline{u}. \end{equation} 
We apply Mean Value Theorem to  get  the inequality 
 $$  (\mu_0 -\mu)(\overline{u} -\underline{u}) \leq  (\mu_0 -\mu) \xi ( \ln(\overline{u})- \ln (\underline{u}) )   $$
 where $\xi \in (\underline{u}, \overline{u}) $ i.e. 
  $$  (\mu_0 -\mu)(\overline{u} -\underline{u}) \leq  (\mu_0 -\mu)\frac{\overline{u}_0}{\underline{u}_0}   ( \ln(\overline{u})- \ln (\underline{u}) )   $$
  which implies 
  $$\frac{d}{dt} \left( \log( \overline{u}) -\log(\underline{u}) \right) \leq     (\mu_0 -\mu)\frac{\overline{u}_0}{\underline{u}_0}    ( \ln(\overline{u})- \ln (\underline{u}) )   .$$
  We solve the previous differential inequality  and  it results 
 $$  \log( \overline{u}) -\log(\underline{u}) \leq  [  \log( \overline{u}_0) -\log(\underline{u}_0)] e^{ (\mu_0 -\mu)\frac{\overline{u}_0}{\underline{u}_0} t}$$
 which implies, in view of (\ref{5.1}) that $\underline{u}$ and $\overline{u}$ are uniformly bounded in time and therefore $T_0= \infty$.  
 We take exponentials in the previous inequality   
 $$  \frac{\overline{u}}{\underline{u}}  \leq exp\{ e^{ (\mu_0 -\mu)\frac{\overline{u}_0}{\underline{u}_0} t}\}$$
 which proves 
  $$  \overline{u} - \underline{u}   \leq \underline{u} ( exp\{ c e^{ (\mu_0 -\mu)\frac{\overline{u}_0}{\underline{u}_0} t}\}-1)
  \leq  ( exp\{ c e^{ (\mu_0 -\mu)\frac{\overline{u}_0}{\underline{u}_0} t}\}-1)$$
  taking limits when $t \rightarrow \infty$ we obtain the wished result. 
     \end{proof} 
     
     {\bf End of the proof of theorem \ref{t1.1}}
     \newline 
    Theorem  \ref{theorem2.1} proves the local existence of solutions. Lemma \ref{l4.1} proves that the solution is bounded by the auxiliary  functions $\overline{u}$ and $\underline{u}$. 
    Lemma \ref{l5.2} gives the existence of global in time solutions as a consequence of the global existence of the upper and lower functions. The asymptotic behavior is obtained in view of 
    $$\|u-1\|_{L^{\infty}(\Omega)} \leq |\overline{u}-1| +    |\underline{u}-1|,$$ 
    $$\|v-1\|_{L^{\infty}(\Omega)} \leq |\overline{u}-1| +    |\underline{u}-1|$$ 
    Lemma \ref{l3.2} and Lemma \ref{l5.2}. .

     {\bf Acknowledgments.}
I want  to express  my deep  gratitude  to  Professor  Ildefonso D\'{\i}az
for his advises, support, comments   and for sharing his  knowledge 
during the last three decades.   Thank you very much Ildefonso, it is always a great pleasure  to learn from you.     
\bibliography{sn-bibliography}

 \end{document}